\documentclass[11pt,reqno]{amsart}

\usepackage[T1]{fontenc}
\usepackage{lmodern}
\usepackage{microtype}
\usepackage{amsmath,amssymb,amsthm}
\usepackage{xcolor}
\usepackage[a4paper,margin=1.18in]{geometry}
\usepackage[
  backend=biber,
  style=numeric,
  sorting=none,
  giveninits=true,
  maxbibnames=99,
  doi=true,
  url=false,
  isbn=false
]{biblatex}
\renewbibmacro{in:}{}
\usepackage[
  colorlinks=true,
  linkcolor=blue!55!black,
  citecolor=blue!55!black,
  urlcolor=blue!55!black
]{hyperref}
\hypersetup{
  pdftitle={An upper bound for an exceptional automorphism group},
  pdfauthor={Xu Zhuang},
}

\newtheorem{theorem}{Theorem}[section]
\newtheorem{proposition}[theorem]{Proposition}
\newtheorem{lemma}[theorem]{Lemma}
\theoremstyle{remark}
\newtheorem{remark}[theorem]{Remark}

\newcommand{\Aut}{\operatorname{Aut}}
\newcommand{\Gal}{\operatorname{Gal}}
\newcommand{\Fix}{\operatorname{Fix}}
\newcommand{\F}{\mathbf{F}}

\title
{An upper bound for an exceptional automorphism group}

\author{Xu Zhuang}
\address{Department of Mathematics, University of California, Irvine,
  Irvine, CA 92697, USA}
\email{xzhuang8@uci.edu}

\date{}

\begin{document}

\begin{abstract}
Let $q=p^h>7$ be odd, put $m=(q+1)/2$, and suppose that $i=(m-2)/2$ satisfies $\gcd(i,m)=\gcd(i+2,m)=1$. For the $\mathbf{F}_{q^2}$-maximal function field $\mathcal{F}_i=\mathbf{F}_{q^2}(x,y)$ defined by $y^m=x^i(x^2+1)$, Peter Beelen, Maria Montanucci, Jonathan Niemann, and Luciane Quoos showed that the geometric automorphism group contains a subgroup of order $4(q+1)$ and conjectured that its order is exactly $4(q+1)$. We prove this equality by establishing the reverse inequality.
\end{abstract}

\maketitle

\section{Introduction}

Maximal function fields over finite fields have many rational places and often have large automorphism groups. These groups are useful in the study of isomorphism classes. For a function field given by an explicit equation, one can frequently construct a large subgroup of automorphisms directly. Showing that the constructed subgroup is the full automorphism group is usually a separate problem and may require a subfield that is preserved by every automorphism.

Let $q=p^h$ be a power of an odd prime and put $m=(q+1)/2$. Beelen, Montanucci, Niemann, and Quoos study the family of $\F_{q^2}$-maximal function fields $\mathcal{F}_i=\F_{q^2}(x,y)$ defined by $y^m=x^i(x^2+1)$,
under the assumptions $\gcd(i,m)=\gcd(i+2,m)=1$. Let $k=\overline{\F}_{q^2}=\overline{\F}_q$ and write $K_i=k\mathcal{F}_i=k(x,y)$ for the extension of constants to $k$. This function field has genus $m-1$; see \cite[Section~3]{BeelenMontanucciNiemannQuoos}. They give a precise description of its geometric automorphism group $\Aut_k(K_i)$ except in the case $q>7$ and $i=(m-2)/2$; see \cite[Theorem~4.9]{BeelenMontanucciNiemannQuoos}. This parameter is fixed by the symmetry $i\mapsto m-2-i$ coming from the reciprocal change of variables in \cite[Lemma~3.2]{BeelenMontanucciNiemannQuoos}, which therefore gives an additional automorphism in the exceptional case.

In this remaining case, they show that $\Aut_k(K_i)$ contains a subgroup of order $4(q+1)$, which we denote by $H$; see \cite[Theorem~4.6]{BeelenMontanucciNiemannQuoos}. They conjecture that $|\Aut_k(K_i)|=4(q+1)$ \cite[Conjecture~4.11]{BeelenMontanucciNiemannQuoos}. Since $H$ already has this order, the conjecture is equivalent to the assertion that $H$ is the full geometric automorphism group.
\begin{theorem}\label{thm:main}
Let $q=p^h>7$ be odd, let $m=(q+1)/2$, and suppose that $i=(m-2)/2$ satisfies $\gcd(i,m)=\gcd(i+2,m)=1$. If $H$ is the subgroup constructed in \cite[Theorem~4.6]{BeelenMontanucciNiemannQuoos}, then $\Aut_k(K_i)=H$. In particular, $|\Aut_k(K_i)|=4(q+1)$.
\end{theorem}

The lower bound in Theorem~\ref{thm:main} is exactly the result of \cite[Theorem~4.6]{BeelenMontanucciNiemannQuoos}. The content of this paper is the upper bound $|\Aut_k(K_i)|\leq4(q+1)$.

The main idea is to use a rational subfield different from $k(x)$. Write $m=2n$, so that the exceptional parameter is $i=n-1$, and consider $u=y^2/x$. The sign involution $\tau:(x,y)\mapsto(x,-y)$ and the reciprocal involution $\rho:(x,y)\mapsto(x^{-1},y/x)$ both fix $u$. We show that they generate the Galois group $V$ of the degree-four extension $K_i/k(u)$. Moreover, $k(u)$ is the unique rational subfield of index four and is preserved by every element of $\Aut_k(K_i)$.

It remains to bound the induced automorphisms of $k(u)$. The group $V$ has three nonidentity involutions. The sign involution $\tau$ fixes four places of $K_i$, whereas $\rho$ and $\tau\rho$ each fix $2n$ places. Every element of $\Aut_k(K_i)$ preserves $k(u)$ and hence conjugates $V$ to itself. Since conjugation preserves the number of fixed places, it must fix $\tau$. The restrictions to $k(u)$ of the four places fixed by $\tau$ are the zero and the pole of $u$. Consequently, every induced automorphism sends $u$ to $cu$ or $c/u$ for some $c\in k^\times$. It also permutes the remaining places of $k(u)$ that ramify in $K_i$, namely the zeros of $u^{2n}-4$. This leaves at most $4n$ induced automorphisms and gives the upper bound $16n=4(q+1)$ for the full automorphism group.

\section{The biquadratic extension and its intermediate fields}

Since $i=(m-2)/2$ is an integer, $m$ is even. Write $m=2n$, so that $i=n-1$ and $q+1=4n$. The coprimality assumptions imply that $n$ is even: if $n$ were odd, both $n-1$ and $n+1$ would have a common factor $2$ with $2n$. Moreover, $n\geq6$. Indeed, the smaller even values $n=2$ and $n=4$ give respectively $q=7$ and $q=15$, and neither is allowed by the hypotheses.

\begin{remark}
Conversely, when $n$ is even, the equalities $\gcd(n-1,2n)=\gcd(n+1,2n)=1$ hold automatically. Since $q+1=4n$, the exceptional parameters satisfying the coprimality assumptions are therefore exactly the odd prime powers $q>7$ with $q\equiv7\pmod{8}$. The first values are $q=23,31,47,71,79,\ldots$.
\end{remark}

Put $K=K_i=k(x,y)$, where $y^{2n}=x^{n-1}(x^2+1)$. All function fields and automorphisms considered below are over $k$. Since $k$ is algebraically closed, every place of these function fields has degree one. For a function $f$ regular at a place $P$, write $f(P)\in k$ for its residue.

Choose $\alpha\in k$ with $\alpha^2=-1$. In $k(x)$, the function $x^{n-1}(x^2+1)$ has zeros of orders $n-1,1,1$ at the zeros of $x,x-\alpha,x+\alpha$, respectively, and a pole of order $n+1$ at the pole of $x$. The coprimality assumptions are $\gcd(n-1,2n)=\gcd(n+1,2n)=1$, so all four orders are coprime to $2n$. Since $p\nmid2n$, the Kummer ramification criterion shows that $K/k(x)$ is cyclic of degree $2n$ and that these four places are totally ramified, with no other ramification. Let $P_0,P_\infty,P_\alpha,P_{-\alpha}$ denote the unique places of $K$ above them. We use the divisor calculation $(x)=2n(P_0-P_\infty)$ and $(y)=(n-1)P_0+P_\alpha+P_{-\alpha}-(n+1)P_\infty$ from \cite[Equation~(3)]{BeelenMontanucciNiemannQuoos}.
They also show that $g(K)=2n-1$.
Set $u=y^2/x$. The preceding divisor formulas give $(u)=2P_\alpha+2P_{-\alpha}-2P_0-2P_\infty$.
Thus the pole divisor of $u$ has degree four, and hence $[K:k(u)]=4$; see \cite[Theorem~1.4.11]{Stichtenoth}. We will also use the identity $u^n=x+x^{-1}$.

Consider the two automorphisms $\tau(x,y)=(x,-y)$ and $\rho(x,y)=(x^{-1},y/x)$.
The second is the reciprocal isomorphism of \cite[Lemma~3.2]{BeelenMontanucciNiemannQuoos}, specialized to the fixed parameter $i=n-1$. The maps $\tau$ and $\rho$ are commuting involutions, and both fix $u$.

\begin{proposition}\label{prop:cover}
The extension $K/k(u)$ is Galois with group $V=\langle\tau,\rho\rangle\simeq C_2\times C_2$.
\end{proposition}

\begin{proof}
The distinct commuting involutions $\tau$ and $\rho$ generate a subgroup
$V$ of order four. Since both fix $u$, we have
$V\leq\Aut_{k(u)}(K)$. Then $[K:k(u)]=4$ gives
$4=|V|\leq|\Aut_{k(u)}(K)|\leq[K:k(u)]=4$.
Thus $\Aut_{k(u)}(K)=V$, and its order equals the degree of the
extension. Hence $K/k(u)$ is Galois with Galois group $V$.
\end{proof}

For $\iota\in\Aut_k(K)$, let $\Fix(\iota)$ denote the set of places of $K$ fixed by $\iota$. Recall that, in a finite Galois extension over the algebraically closed field $k$, the stabilizer of a place equals its inertia group, since the residue field extension is trivial. Its order is therefore the ramification index. We use this fact to determine the fixed places of the three nontrivial elements of $V$.

\begin{lemma}\label{lem:fixed}
The involution $\tau$ fixes precisely $P_0,P_\infty,P_\alpha,P_{-\alpha}$. The fixed places of $\rho$ are the $2n$ places above the zero of $x-1$, and the fixed places of $\tau\rho$ are the $2n$ places above the zero of $x+1$. In particular, $|\Fix(\tau)|=4$ and $|\Fix(\rho)|=|\Fix(\tau\rho)|=2n$.
\end{lemma}

\begin{proof}
The four places $P_0,P_\infty,P_\alpha,P_{-\alpha}$ are totally ramified in $K/k(x)$, so they are fixed by every element of $\Gal(K/k(x))$. Every other place is unramified and hence has trivial stabilizer in this group. Since $\tau\in\Gal(K/k(x))$, it fixes precisely these four places.

If a place $P$ of $K$ is fixed by $\rho$, its restriction to $k(x)$ is fixed by the automorphism $x\mapsto x^{-1}$. The zero and pole of $x$ are exchanged, so $x(P)$ is finite and nonzero and must satisfy $x(P)=x(P)^{-1}$. Thus $x(P)=1$ or $x(P)=-1$. Each of these two places of $k(x)$ is unramified and splits into $2n$ places of $K$. For $x(P)=1$, these places are distinguished by the $2n$ distinct residues $y(P)$ satisfying $y(P)^{2n}=2$. The formula $\rho(y)=y/x$ preserves each such residue, so $\rho$ fixes all of these places. For $x(P)=-1$, the residues satisfy $y(P)^{2n}=-2\ne0$, and $\rho$ changes $y(P)$ to $-y(P)$, so it fixes none of these places. For $\tau\rho$, the formula $(\tau\rho)(y)=-y/x$ gives the reverse conclusion: it fixes all places with $x(P)=-1$ and none with $x(P)=1$.
\end{proof}

\begin{proposition}\label{prop:quotient-genus}
The fixed field $K^{\langle\tau\rangle}$ has genus $n-1$, while
$K^{\langle\rho\rangle}$ and $K^{\langle\tau\rho\rangle}$ each have genus
$n/2$. In particular, each of the three intermediate quadratic fields of $K/k(u)$ has genus at least three.
\end{proposition}

\begin{proof}
For $\iota\in\{\tau,\rho,\tau\rho\}$, put $E=K^{\langle\iota\rangle}$. The extension $K/E$ is Galois of degree two. A place $P$ of $K$ is ramified in $K/E$ if and only if it is fixed by $\iota$, in which case its ramification index is two. Since $\operatorname{char}k\ne2$, the extension is tame, and the different exponent at each ramified place is $2-1=1$. All places have degree one, so the degree of the different is $r=|\Fix(\iota)|$. The Hurwitz genus formula therefore gives
$2g(K)-2=2(2g(E)-2)+r$; see
\cite[Theorem~3.4.13]{Stichtenoth}.

Now $2g(K)-2=4n-4$. For $\iota=\tau$, Lemma~\ref{lem:fixed} gives
$r=4$, and hence $4n-4=2(2g(E)-2)+4$, so $g(E)=n-1$.
For $\iota=\rho$ or $\tau\rho$, it gives $r=2n$, and hence
$4n-4=2(2g(E)-2)+2n$, so $g(E)=n/2$. Thus each fixed field has
genus at least three because $n\geq6$.
\end{proof}

\section{The unique rational subfield of index four}

We use Castelnuovo's inequality in the following form. If $F_1,F_2\subset L$, if $L=F_1F_2$, and if $d_j=[L:F_j]$, then
\begin{equation}\label{eq:CS}
  g(L)\leq d_1g(F_1)+d_2g(F_2)+(d_1-1)(d_2-1).
\end{equation}
See \cite[Theorem~3.11.3]{Stichtenoth}.

The first consequence rules out subfields of index two and genus at most one, which could otherwise arise in the uniqueness argument.

\begin{proposition}\label{prop:no-low-genus}
There is no subfield $E$ of $K$ containing $k$ such that $[K:E]=2$ and $g(E)\leq1$.
\end{proposition}

\begin{proof}
Suppose, to the contrary, that such a subfield $E$ exists. Since $\operatorname{char}(k)\ne2$, the extension $K/E$ is separable and, being quadratic, is Galois. Put $B=k(u)$, so $[K:B]=4$. Since $E\subseteq EB\subseteq K$ and $[K:E]=2$, either $EB=E$ or $EB=K$.

If $EB=K$, then $B=k(u)$ is rational, so $g(B)=0$. Applying \eqref{eq:CS} to $E$ and $B$ gives $g(K)\leq2g(E)+4g(B)+(2-1)(4-1)=2g(E)+3\leq5$, contrary to $g(K)=2n-1\geq11$.

It follows that $EB=E$, so $B\subseteq E$. Let $\iota$ be the nontrivial automorphism of the quadratic Galois extension $K/E$. Since $\iota$ fixes $B$, Proposition~\ref{prop:cover} gives $\iota\in V$. Hence $E=K^{\langle\iota\rangle}$ is one of the three intermediate quadratic fields of $K/B$. Proposition~\ref{prop:quotient-genus} shows that each of these fields has genus at least three, a contradiction.
\end{proof}

\begin{proposition}\label{prop:unique}
If $v\in K\setminus k$ satisfies $[K:k(v)]=4$, then $k(v)=k(u)$. Equivalently, $k(u)$ is the unique rational subfield of $K$ having index four.
\end{proposition}

\begin{proof}
Let $B=k(u)$ and $B'=k(v)$, set $L=BB'$, and put $e=[K:L]$. Since $[K:B]=[K:B']=4$, the integer $e$ belongs to $\{1,2,4\}$.

If $e=1$, then $K=BB'$. Applying \eqref{eq:CS} to the two rational fields $B$ and $B'$ gives $g(K)\leq(4-1)^2=9$, contrary to $g(K)\geq11$.

Suppose that $e=2$. The tower law gives $[L:B]=[L:B']=2$, and $L=BB'$. Applying \eqref{eq:CS} inside $L$ gives $g(L)\leq(2-1)^2=1$. Since $[K:L]=2$, this contradicts Proposition~\ref{prop:no-low-genus}.

Therefore $e=4$. Since $B\subseteq L$ and both fields have index four in $K$, one has $L=B$. The same argument gives $L=B'$, and hence $B=B'$.
\end{proof}

\section{Proof of the main theorem}

\begin{proof}[Proof of Theorem~\ref{thm:main}]
Put $A=\Aut_k(K)$ and $B=k(u)$. For every $\sigma\in A$, the field $\sigma(B)=k(\sigma(u))$ is rational and satisfies $[K:\sigma(B)]=[K:B]=4$. Proposition~\ref{prop:unique} gives $\sigma(B)=B$. Thus restriction defines a homomorphism $\operatorname{res}:A\to\Aut_k(B)$ by $\operatorname{res}(\sigma)=\sigma|_B$. Its kernel is $\Aut_B(K)=V$ by Proposition~\ref{prop:cover}. In particular, $V$ is normal in $A$. Write $\Gamma=\operatorname{res}(A)$. We will show that $|\Gamma|\leq4n$; since each element of $\Gamma$ has exactly $|V|=4$ preimages in $A$, this will give the required upper bound.

For $a\in k$, let $Q_a$ be the zero of $u-a$ in $B$, and let $Q_\infty$ be the pole of $u$. We first determine the places of $B$ that ramify in $K$. A place of $K$ is ramified in the Galois extension $K/B$ if and only if its stabilizer in $V$ is nontrivial, or equivalently if it is fixed by one of $\tau,\rho,\tau\rho$. By the divisor of $u$, the restrictions to $B$ of the places fixed by $\tau$ are precisely $Q_0$ and $Q_\infty$.

At a place $P$ fixed by $\rho$, Lemma~\ref{lem:fixed} gives $x(P)=1$, so $u(P)^n=2$. Conversely, if $a^n=2$, choose $b\in k$ with $b^2=a$. Then $b^{2n}=2$, and the place with residues $x(P)=1$ and $y(P)=b$ is fixed by $\rho$ and restricts to $Q_a$. At a place fixed by $\tau\rho$, one has $x(P)=-1$ and $u(P)^n=-2$. Conversely, if $a^n=-2$, choose $b$ with $b^2=-a$. Since $n$ is even, $b^{2n}=(-a)^n=-2$, so the place with residues $x(P)=-1$ and $y(P)=b$ is fixed by $\tau\rho$ and restricts to $Q_a$. Consequently, the places of $B$ that ramify in $K$ are exactly
\[
  \{Q_0,Q_\infty\}\ \cup\ \mathcal{R},\qquad
  \mathcal{R}=\{Q_a:a\in k,\ a^{2n}=4\}.
\]
Here $\mathcal{R}$ is the set of zeros of $u^{2n}-4$ in $B$. The two subsets defined by $a^n=2$ and $a^n=-2$ are disjoint, since the characteristic is odd.

For $\sigma\in A$ and $\iota\in V$, one has $\sigma(\Fix(\iota))=\Fix(\sigma\iota\sigma^{-1})$, so conjugation preserves the number of fixed places. By Lemma~\ref{lem:fixed}, $\tau$ is the unique nonidentity element of $V$ with four fixed places, since $2n>4$. Normality of $V$ therefore gives $\sigma\tau\sigma^{-1}=\tau$ for every $\sigma\in A$. It follows that $\sigma$ permutes $\Fix(\tau)$, and hence $\sigma|_B$ permutes the pair $\{Q_0,Q_\infty\}$. Moreover, conjugation by $\sigma$ permutes $\rho$ and $\tau\rho$, so $\sigma$ also preserves $\Fix(\rho)\cup\Fix(\tau\rho)$. Restricting these places to $B$ shows that $\sigma|_B$ permutes $\mathcal{R}$.

Let $\gamma\in\Gamma$. The divisor of $u$ in $B$ is $(u)=Q_0-Q_\infty$. Since $\gamma$ permutes $Q_0$ and $Q_\infty$, the divisor of $\gamma(u)$ is either $Q_0-Q_\infty$ or $Q_\infty-Q_0$. A nonzero function whose principal divisor is zero belongs to $k^\times$, so respectively $\gamma(u)/u\in k^\times$ or $\gamma(u)u\in k^\times$. Hence
\[
  \gamma(u)=cu\quad\hbox{or}\quad\gamma(u)=c/u,
  \qquad c\in k^\times.
\]

Finally, $\gamma$ permutes $\mathcal{R}$, the set of zeros of $u^{2n}-4$. Therefore the set of zeros of $\gamma(u^{2n}-4)$ is again $\mathcal{R}$. If $\gamma(u)=cu$, then $\gamma(u^{2n}-4)=c^{2n}u^{2n}-4$. Evaluating at any $Q_a\in\mathcal{R}$ gives $4c^{2n}-4=0$, so $c^{2n}=1$. If $\gamma(u)=c/u$, then $\gamma(u^{2n}-4)=c^{2n}/u^{2n}-4$, and evaluation at $Q_a$ gives $c^{2n}/4-4=0$, so $c^{2n}=16$. Since $k$ is algebraically closed and $p\nmid2n$, each of the equations $c^{2n}=1$ and $c^{2n}=16$ has exactly $2n$ distinct solutions in $k$. An automorphism of $B=k(u)$ is determined by its value on $u$. There are therefore at most $4n$ possibilities for $\gamma$, and
\[
  |A|=|V|\,|\Gamma|\leq4\cdot4n=16n=4(q+1).
\]
By \cite[Theorem~4.6]{BeelenMontanucciNiemannQuoos}, $A$ contains the subgroup $H$ of order $4(q+1)$. Hence equality holds and $\Aut_k(K_i)=H$.
\end{proof}

\section*{Acknowledgments}
The author thanks Yuxiang Yao for helpful discussions and Maria Montanucci for suggestions that improved the exposition and organization of the paper.

\section*{Statement on AI-assisted preparation}
During the preparation of this manuscript, the author used OpenAI's GPT-5.6 Sol
language model through Codex for exploratory proof development, literature
organization, and language editing. In particular, the suggestion to organize
the upper-bound argument around the common invariant $u=y^2/x$ arose during
that interaction. The author independently checked every argument 
and assumes full responsibility for the content of the manuscript.

\printbibliography[title={References}]

\end{document}